\documentclass{amsart}

\usepackage{amsmath}
\usepackage{amsfonts}
\usepackage{amssymb,enumerate}
\usepackage{amsthm}
\usepackage[all]{xy}
\usepackage{hyperref}

\newtheorem{lem}{Lemma}[section]
\newtheorem{cor}[lem]{Corollary}

\newtheorem{thm}[lem]{Theorem}

\newtheorem{Defn}[lem]{Definition}
\newtheorem{Ex}[lem]{Example}
\newtheorem{Question}[lem]{Question}
\newtheorem{Property}[lem]{Property}
\newtheorem{Properties}[lem]{Properties}
\newtheorem{Discussion}[lem]{Remark}
\newtheorem{Construction}[lem]{Construction}
\newtheorem{Notation}[lem]{Notation}
\newtheorem{Fact}[lem]{Fact}
\newtheorem{Notationdefinition}[lem]{Definition/Notation}
\newtheorem{Remarkdefinition}[lem]{Remark/Definition}
\newtheorem{Subprops}{}[lem]
\newtheorem{Para}[lem]{}

\newenvironment{defn}{\begin{Defn}\rm}{\end{Defn}}

\newenvironment{fact}{\begin{Fact}\rm}{\end{Fact}}

\newenvironment{para}{\begin{Para}\rm}{\end{Para}}
\newenvironment{disc}{\begin{Discussion}\rm}{\end{Discussion}}

\newtheorem{intthm}{Theorem}

\newcommand{\cat}[1]{\mathcal{#1}}
\newcommand{\catx}{\cat{X}}

\newcommand{\catp}{\cat{P}}
\newcommand{\catf}{\cat{F}}
\newcommand{\cati}{\cat{I}}

\newcommand{\catb}{\cat{B}}
\newcommand{\catgi}{\cat{GI}}
\newcommand{\catgp}{\cat{GP}}

\newcommand{\catgic}{\cat{GI}_C}

\newcommand{\catgpc}{\cat{GP}_C}

\newcommand{\catbc}{\cat{B}_C}

\newcommand{\catgfc}{\cat{GF}_C}
\newcommand{\catic}{\cat{I}_C}

\newcommand{\catpc}{\cat{P}_C}
\newcommand{\catfc}{\cat{F}_C}

\newcommand{\pd}{\operatorname{pd}}

\newcommand{\id}{\operatorname{id}}	
\newcommand{\fd}{\operatorname{fd}}

\newcommand{\catpd}[1]{\cat{#1}\text{-}\pd}
\newcommand{\xpd}{\catpd{X}}

\newcommand{\xid}{\catid{X}}

\newcommand{\gpd}{\catpd{G}}
\newcommand{\gid}{\catid{G}}
\newcommand{\catid}[1]{\cat{#1}\text{-}\id}

\newcommand{\pcpd}{\catpc\text{-}\pd}
\newcommand{\fcpd}{\catfc\text{-}\pd}

\newcommand{\icid}{\catic\text{-}\id}

\newcommand{\gpcpd}{\catgpc\text{-}\pd}

\newcommand{\gfcpd}{\catgfc\text{-}\pd}

\newcommand{\gicid}{\catgic\text{-}\id}

\newcommand{\depth}{\operatorname{depth}}	
\newcommand{\width}{\operatorname{width}}

\newcommand{\ext}{\operatorname{Ext}}	
\newcommand{\rhom}{\mathbf{R}\!\operatorname{Hom}}	

\newcommand{\HH}{\operatorname{H}}
\newcommand{\Hom}{\operatorname{Hom}}	

\newcommand{\spec}{\operatorname{Spec}}

\newcommand{\tor}{\operatorname{Tor}}
\newcommand{\im}{\operatorname{Im}}

\newcommand{\ideal}[1]{\mathfrak{#1}}
\newcommand{\m}{\ideal{m}}

\newcommand{\p}{\ideal{p}}
\newcommand{\q}{\ideal{q}}

\newcommand{\comp}[1]{\widehat{#1}}

\newcommand{\ass}{\operatorname{Ass}}
\newcommand{\supp}{\operatorname{Supp}}

\newcommand{\bbz}{\mathbb{Z}}

\newcommand{\xra}{\xrightarrow}

\renewcommand{\geq}{\geqslant}
\renewcommand{\leq}{\leqslant}

\renewcommand{\hom}{\Hom}

\newcommand{\caticr}{\cat{I}_C(R)}
\newcommand{\catpcr}{\cat{P}_C(R)}
\newcommand{\catfcr}{\cat{F}_C(R)}
\newcommand{\caticc}{\cat{I}_C}
\newcommand{\catpcc}{\cat{P}_C}

\newcommand{\catii}{\cat{I}}

\newcommand{\catpp}{\cat{P}}

\newcommand{\gfd}{\operatorname{Gfd}}
\renewcommand{\gid}{\operatorname{Gid}}

\newcommand{\catacr}{\cat{A}_C(R)}
\newcommand{\catbcr}{\cat{B}_C(R)}
\newcommand{\gc}{\text{G}_C}

\newcommand{\catpcp}{\cat{P}_{C_{\p}}}
\newcommand{\caticp}{\cat{I}_{C_{\p}}}
\newcommand{\pcppd}{\catpcp\text{-}\pd}
\newcommand{\icpid}{\caticp\text{-}\id}
\renewcommand{\gpd}{\operatorname{Gpd}}

\numberwithin{equation}{lem}

\begin{document}

\bibliographystyle{amsplain}

\author{Sean Sather-Wagstaff}

\address{Sean Sather-Wagstaff,
Department of Mathematics,
NDSU Dept \# 2750,
PO Box 6050,
Fargo, ND 58108-6050 USA}

\email{Sean.Sather-Wagstaff@ndsu.edu}

\urladdr{http://math.ndsu.nodak.edu/faculty/ssatherw/}

\author{Siamak Yassemi}
\address{Siamak Yassemi,
Department of Mathematics, University of Tehran, P.O. Box 13145--448, Tehran, Iran, and School of Mathematics, Institute for Research in Fundamental Sciences (IPM), P.O. Box 19395-5746, Tehran, Iran}

\email{yassemi@ipm.ir}

\urladdr{http://math.ipm.ac.ir/yassemi/}

\thanks{This research 
was conducted while S.S.-W.\  visited the IPM in Tehran during July 2008.  
The research of S.Y.\ was supported in part by a grant from the IPM (No. 87130211).}

\title[Modules of finite homological dimension]
{Modules of finite homological dimension with respect to a semidualizing module}




\begin{abstract}
We prove versions of results of Foxby and Holm about modules of
finite (Gorenstein) injective dimension
and finite (Gorenstein) projective dimension with respect to a
semidualizing module. We also verify special cases of a question of 
Takahashi and White.
\end{abstract}

\maketitle

\section*{Introduction} \label{sec00}

Let $R$ be a commutative noetherian ring.
It is well-known that, if $R$ is Gorenstein and local, then 
every module with finite projective dimension has finite injective dimension. 
Conversely, Foxby~\cite{foxby:ibcahtm,foxby:bcfm} showed that,
if $R$ is local and admits a
finitely generated module of finite projective dimension and finite
injective dimension, then $R$ is Gorenstein. 
More recently, Holm~\cite{holm:rfgid} 
proved that, if  $M$ is an $R$-module of finite projective
dimension and finite \emph{Gorenstein} injective dimension, then 
$M$ has finite injective dimension, and so the localization $R_{\p}$ is 
Gorenstein for each $\p\in\spec(R)$
with $\depth_{R_{\p}}(M_{\p})<\infty$. See Section~\ref{sec01} for 
terminology and notation.

In this paper, we prove analogues of these results for homological
dimensions defined in terms of semidualizing $R$-modules. 
For instance, the following result is proved in~\eqref{para0201}.
Other variants of this result are also given in Section~\ref{sec02}.
It should be noted that our proof of this result
is different from Holm's proof for the special case $C=R$.
In particular, this paper also provides a new proof of Holm's result.

\begin{intthm} \label{thm0001}
Let 
$C$ be a semidualizing $R$-module, and 
let $M$ be an $R$-module
with $\pcpd_R(M)<\infty$ and $\gid_R(M)<\infty$. Then
$\id_R(M)=\gid_R(M)<\infty$ and, for each $\p\in\spec(R)$
with $\depth_{R_{\p}}(M_{\p})$ finite, the $R_{\p}$-module
$C_{\p}$ is  dualizing. 
\end{intthm}

Takahashi and White~\cite{takahashi:hasm} posed the following question: 
When $R$ is a local Cohen-Macaulay ring admitting a dualizing
module and $C$ is a semidualizing $R$-module, if $M$ is an $R$-module
of finite depth
such that $\pcpd_R(M)$ and $\icid_R(M)$ are finite,
must $R$ be  Gorenstein?
An affirmative answer to this question would yield another generalization
of Foxby's theorem.
Our techniques allow us to answer the question in the affirmative
in three special cases. The first one is contained in the next result
which we prove in~\eqref{para0202}; the others are in Theorem~\ref{thm0204}
and Corollary~\ref{cor0205}.

\begin{intthm} \label{thm0002}
Let
$C$ be a semidualizing $R$-module, and let $M$ be an $R$-module
with $\pcpd_R(M)=0$ and $\icid_R(M)<\infty$. Then
$R_{\p}$ is  Gorenstein for each $\p\in\supp_R(M)$.
\end{intthm}

\section{Semidualizing Modules and Related Homological Dimensions} \label{sec01}

Throughout this paper $R$ is a commutative noetherian ring.

This section contains definitions and background information for use in
the proofs of our main results in
Section~\ref{sec02}.

\begin{defn}
Let $\catx$ be a class of $R$-modules and $M$ an $R$-module.
An \emph{$\catx$-resolution} of $M$ is a complex of $R$-modules in $\catx$
of the form
$$X =\cdots\xra{\partial^X_{n+1}}X_n\xra{\partial^X_n}
X_{n-1}\xra{\partial^X_{n-1}}\cdots\xra{\partial^X_{1}}X_0\to 0$$
such that $\HH_0(X)\cong M$ and $\HH_n(X)=0$ for $n\geq 1$.
The \emph{$\catx$-projective dimension} of $M$ is the quantity
$$\xpd_R(M)=\inf\{\sup\{n\geq 0\mid X_n\neq 0\}\mid \text{$X$ is an
$\catx$-resolution of $M$}\}.$$
In particular, one has $\xpd_R(0)=-\infty$.
The modules of $\catx$-projective dimension 0 are
the nonzero modules of $\catx$.

Dually, an \emph{$\catx$-coresolution} of $M$ is a complex of $R$-modules in $\catx$
of the form
$$X =\quad 0\to X_0\xra{\partial^X_{0}}X_{-1}\xra{\partial^X_{-1}}\cdots\xra{\partial^X_{n+1}}
X_{n}\xra{\partial^X_{n}}X_{n-1}\xra{\partial^X_{n-1}}\cdots$$
such that $\HH_0(X)\cong M$ and $\HH_n(X)=0$ for $n\leq -1$.
The \emph{$\catx$-injective dimension} of $M$ is the quantity
$$\xid_R(M)=\inf\{\sup\{-n\geq 0\mid X_n\neq 0\}\mid \text{$X$ is an
$\catx$-coresolution of $M$}\}.$$
In particular, one has $\xid_R(0)=-\infty$.
The modules of $\catx$-injective dimension 0 are
the nonzero modules of $\catx$.

When $\catx$ is the class of projective $R$-modules, we write
$\pd_R(M)$ for the associated homological dimension and call it the
\emph{projective dimension} of $M$. Similarly, the \emph{flat 
and injective dimensions} of $M$ are denoted $\fd_R(M)$ and $\id_R(M)$,
respectively.
\end{defn}

The homological dimensions of interest in this paper are built from
semidualizing modules and their associated projective and injective classes, defined next. 
Semidualizing modules occur in the literature with several different names,
e.g., in the work of Foxby~\cite{foxby:gmarm}, Golod~\cite{golod:gdagpi},
Mantese and Reiten~\cite{mantese:wtm},
Vasconcelos~\cite{vasconcelos:dtmc} and Wakamatsu~\cite{wakamatsu:mtse}.
The prototypical semidualizing modules are the dualizing (or canonical)
modules of Grothendieck and Hartshorne~\cite{hartshorne:lc}.

\begin{defn} \label{defn0201}
A finitely generated $R$-module $C$ is \emph{semidualzing} 
if 
the natural homothety morphism
$R\to \Hom_R(C,C)$ is an isomorphism and
$\ext^{\geq 1}_R(C,C)=0$.
An $R$-module $D$ is \emph{dualizing} if it is 
semidualizing and has finite injective dimension.

Let $C$ be a semidualizing $R$-module.
We set
\begin{align*}
\catpcr&=\text{the subcategory of modules $P\otimes_R C$ where $P$ is $R$-projective}\\
\catfcr&=\text{the subcategory of modules $F\otimes_R C$ where $F$ is $R$-flat}\\
\caticr&=\text{the subcategory of modules $\Hom_R(C,I)$ where $I$ is $R$-injective.}
\end{align*}
Modules in $\catpcr$ are called \emph{$C$-projective}, and those
in $\caticr$ are called \emph{$C$-injective}.
\end{defn}

\begin{fact}\label{fact0101}
Let $C$ be a
semidualizing $R$-module.  
It is straightforward to show that, if $P\in\catpc(R)$ and $I\in\catic(R)$,
then $P_{\p}\in\catp_{C_{\p}}(R_{\p})$ and $I_{\p}\in\cati_{C_{\p}}(R_{\p})$
for each $\p\in\spec(R)$. It follows that we have
$\pcppd_{R_{\p}}(M_{\p})\leq\pcpd_R(M)$
and $\icpid_{R_{\p}}(M_{\p})\leq\icid_R(M)$ for each $R$-module $M$.
\end{fact}

\begin{fact}\label{fact0101'}
A result of Gruson and 
Raynaud~\cite[Seconde Partie, Thm.~(3.2.6)]{raynaud:cpptpm}, and 
Jensen~\cite[Prop.~6]{jensen:vl} says that, if $F$ is a flat $R$-module,
then $\pd_R(F)\leq\dim(R)$. It follows that, if $\fcpd_R(M)<\infty$, then
$\pcpd_R(M)<\infty$. 
\end{fact}

The next classes are central to our proofs and 
were introduced  by Foxby~\cite{foxby:gdcmr}.

\begin{defn} \label{notation08d}
Let $C$ be a
semidualizing $R$-module.  
The \emph{Auslander class} of $C$ is the class $\catacr$
of $R$-modules $M$ such that 
\begin{enumerate}[\quad(1)]
\item $\tor^R_{\geq 1}(C,M)=0=\ext_R^{\geq 1}(C,C\otimes_R M)$, and
\item the natural map $M\to\Hom_R(C,C\otimes_R M)$ is an isomorphism.
\end{enumerate}
The \emph{Bass class} of $C$ is the class $\catbc(R)$
of $R$-modules $M$ such that 
\begin{enumerate}[\quad(1)]
\item $\ext_R^{\geq 1}(C,M)=0=\tor^R_{\geq 1}(C,\Hom_R(C,M))$, and 
\item the natural evaluation map $C\otimes_R\Hom_R(C,M)\to M$ is an isomorphism.
\end{enumerate}
\end{defn}

\begin{fact}\label{projac}
Let $C$ be a
semidualizing $R$-module.  
The categories $\catacr$ and $\catbcr$ are closed under extensions,
kernels of epimorphisms and cokernels of monomorphism; see~\cite[Cor.\ 6.3]{holm:fear}.
The category $\catacr$ contains all modules of finite
flat dimension and those of finite 
$\catic$-injective dimension, and the category $\catbcr$ contains all modules of
finite injective dimension and those of finite 
$\catpcc$-projective dimension
by~\cite[Cors.\ 6.1 and 6.2]{holm:fear}.
\end{fact}

The next definitions are due to Holm and J\o rgensen~\cite{holm:smarghd} 
in this generality.  

\begin{defn} \label{defn0301}
Let $C$ be a semidualizing $R$-module.
A \emph{complete $\caticc\catii$-resolution} is a complex $Y$ of $R$-modules 
satisfying the following:
\begin{enumerate}[\quad(1)]
\item $Y$ is exact and $\Hom_R(I,Y)$ is exact for each $I\in\caticr$, and
\item $Y_i\in \caticr$ for all $i\geq 0$ and $Y_i\in\cati(R)$ for all  $i< 0$.
\end{enumerate}
An $R$-module $H$ is \emph{$\text{G}_C$-injective} if there
exists a complete $\caticc\catii$-resolution $Y$ such that $H\cong\im(\partial^Y_0)$,
in which case $Y$ is a \emph{complete $\caticc\catii$-resolution of $H$}.  
We set
$$\catgic(R)=\text{the class of $\text{G}_C$-injective $R$-modules}.$$
In the special case $C=R$, we set
$\gid_R(M)=\catgi_R\text{-}\id_R(M)$,
and we write ``complete injective resolution'' instead of 
``complete $\catii_R\catii$-resolution''.

A \emph{complete $\catpp\catpcc$-resolution} is a complex $X$ of $R$-modules 
satisfying the following.
\begin{enumerate}[\quad(1)]
\item $X$ is exact and $\Hom_R(X,P)$ is exact for each $P\in\catpcr$, and
\item $X_i\in\catp(R)$ for all  $i\geq 0$ and $X_i\in\catpcr$ for all $i< 0$.
\end{enumerate}
An $R$-module $M$ is \emph{$\text{G}_C$-projective} if there
exists a complete $\catpp\catpcc$-resolution $X$ such that $M\cong\im(\partial^X_0)$,
in which case $X$ is a \emph{complete $\catpp\catpcc$-resolution of $M$}.  We set
$$\catgpc(R)=\text{the class of $\text{G}_C$-projective $R$-modules}.$$
In the case $C=R$, we set
$\gpd_R(M)=\catgp_R\text{-}\pd_R(M)$.
\end{defn}

The next two lemmas are proved as in~\cite[(2.17),(2.18)]{christensen:ogpifd}
using tools from~\cite{white:gpdrsm}.

\begin{lem} \label{lem0101}
Let $C$ be a semidualizing $R$-module and let $M$ be an $R$-module with $\gpcpd_R(M)<\infty$.
There is an exact sequence of $R$-modules
$$0\to M\to P\to M'\to 0$$
such that $M'\in\catgpc(R)$ and $\pcpd_R(P)=\gpcpd_R(M)$.
\qed
\end{lem}

\begin{lem} \label{lem0102}
Let $C$ be a semidualizing $R$-module and let $M$ be an $R$-module with $\gicid_R(M)<\infty$.
There is an exact sequence of $R$-modules
$$0\to M'\to E \to M\to 0$$
such that  $\icid_R(E)=\gicid_R(M)$ and $M'\in\catgic(R)$.
\qed
\end{lem}

\begin{defn} \label{defn0101}
Assume that $R$ is local with residue field $k$. The \emph{depth} of a
(not necessarily finitely generated) $R$-module $M$ is
$$\depth_R(M)=\inf\{n\geq 0\mid\ext^n_R(k,M)\neq 0\}.$$
\end{defn}

\section{Main Results} \label{sec02}

\begin{para} \label{para0201}
\emph{Proof of Theorem~\ref{thm0001}.}
As $\gid_R(M)$ is finite,
Lemma~\ref{lem0102} yields an exact sequence of $R$-modules
\begin{equation} \label{para0201a} \tag{$\ast$}
0\to M'\to E\to M\to 0
\end{equation}
such that $\id_R(E)<\infty$ and $M'$ is G-injective.
The finiteness of $\pcpd_R(M)$ and $\id_R(E)$
implies that $M,E\in\catbc(R)$, and so 
$M'\in\catbcr$; see Fact~\ref{projac}.

We claim that $\ext^{\geq 1}_R(M,M')=0$. To see this, let $Y$ be a complete injective resolution
of $M'$
and set $M^{(i)}=\im(\partial_i^Y)$ for each $i\in\bbz$. 
Since $M',Y_i\in\catbcr$ for each $i\in\bbz$, we have $M^{(i)}\in\catbcr$
for each $i$, and so
$\ext^{\geq 1}_R(C,M^{(i)})=0$. Hence
$$\ext^{\geq 1}_R(P\otimes_RC,M^{(i)})\cong\hom_R(P,\ext^{\geq 1}_R(C,M^{(i)}))=0$$
for each projective $R$-module $P$ and each $i$. Using a bounded 
$\catpc$-resolution of $M$, a dimension-shifting argument shows that
$\ext^{\geq d+1}_R(M,M^{(i)})=0$ for each $i$
where $d=\pcpd_R(M)$. Another dimension-shifting
argument using the complete injective resolution of $M'$ yields the following
$$\ext^{\geq 1}_R(M,M')\cong \ext^{\geq 1}_R(M,M^{(0)})
\cong \ext^{\geq d+1}_R(M,M^{(d)})=0$$
as claimed.

The previous paragraph  shows that the sequence~\eqref{para0201a}
splits. Hence, we have 
$$\sup\{\id_R(M),\id_R(M')\}=\id_R(E)<
\infty$$
and so $\id_R(M)<\infty$. The equality $\id_R(M)=\gid_R(M)$ now follows from
the result dual to~\cite[(2.27)]{holm:ghd}.

Now, let $\p\in\spec(R)$ with $\depth_{R_{\p}}(M_{\p})$ finite. 
Using Fact~\ref{fact0101} we conclude  that
$\pcppd_{R_{\p}}(M_{\p})$ and $\id_{R_{\p}}(M_{\p})$ are finite.
The finiteness of $\pcppd_{R_{\p}}(M_{\p})$ implies $M_{\p}\in\catb_{C_{\p}}(R_{\p})$
and thus
$\ext^{\geq 1}_{R_{\p}}(C_{\p},M_{\p})=0$. 
(Hence, in the derived category $\cat D(R_{\p})$, there is an isomorphism
$\rhom_{R_{\p}}(C_{\p},M_{\p})\simeq\hom_{R_{\p}}(C_{\p},M_{\p})$.)
Using~\cite[(2.11.c)]{takahashi:hasm}, the finiteness of $\pcppd_{R_{\p}}(M_{\p})$ also implies
$$\fd_{R_{\p}}(\hom_{R_{\p}}(C_{\p},M_{\p}))
\leq\pd_{R_{\p}}(\hom_{R_{\p}}(C_{\p},M_{\p}))=\pcppd_{R_{\p}}(M_{\p})<\infty.$$
The $R_{\p}$-module $M_{\p}$ has finite injective dimension
and finite depth, so the finiteness of 
$\fd_{R_{\p}}(\hom_{R_{\p}}(C_{\p},M_{\p}))$
implies that $C_{\p}$ is dualizing for $R_{\p}$; see~\cite[(8.2)]{christensen:scatac}.
\qed
\end{para}

\begin{cor} \label{cor0201}
Assume that $R$ is local, and 
let $C$ be a semidualizing $R$-module. The following conditions are equivalent:
\begin{enumerate}[\rm(i)]
\item \label{cor0201a}
$C$ is a dualizing $R$-module;
\item \label{cor0201b}
there exists a finitely generated $R$-module $M\neq 0$ such that
$\pcpd_R(M)<\infty$ and $\id_R(M)<\infty$;
\item \label{cor0201c}
there exists an $R$-module $M\neq 0$ of finite depth
such that
$\pcpd_R(M)<\infty$ and $\gid_R(M)<\infty$.
\end{enumerate}
\end{cor}

\begin{proof}
The implication \eqref{cor0201b}$\implies$\eqref{cor0201c}
is straightforward, and \eqref{cor0201c}$\implies$\eqref{cor0201a}
follows from Theorem~\ref{thm0001}. 
For \eqref{cor0201a}$\implies$\eqref{cor0201b}, note that 
$\pcpd_R(C)<\infty$ and $\id_R(C)<\infty$
since $C$ is dualizing for $R$. 
\end{proof}

The following versions of Theorem~\ref{thm0001} and Corollary~\ref{cor0201}
are proved similarly, using Lemmas~\ref{lem0101} and~\ref{lem0102}. 

\begin{thm} \label{thm0202}
Let 
$C$ a semidualizing $R$-module, and 
let $M$ be an $R$-module
with $\pd_R(M)<\infty$ and $\gicid_R(M)<\infty$. Then
$\icid_R(M)=\gicid_R(M)<\infty$. Furthermore, for each $\p\in\spec(R)$
such that $\depth_{R_{\p}}(M_{\p})$ is finite, the localization
$C_{\p}$ is a dualizing 
$R_{\p}$-module. \qed
\end{thm}

\begin{cor} \label{cor0203}
Assume that $R$ is local, and 
let $C$ be a semidualizing $R$-module. The following conditions are equivalent:
\begin{enumerate}[\rm(i)]
\item \label{cor0203a}
$C$ is a dualizing $R$-module;
\item \label{cor0203b}
there exists a finitely generated $R$-module $M\neq 0$ such that
$\pd_R(M)<\infty$ and $\icid_R(M)<\infty$;
\item \label{cor0203c}
there exists an $R$-module $M\neq 0$ of finite depth
such that
$\pd_R(M)<\infty$ and $\gicid_R(M)<\infty$. \qed
\end{enumerate}
\end{cor}

\begin{thm} \label{thm0201}
Let 
$C$ a semidualizing $R$-module, and 
let $M$ be an $R$-module
with $\icid_R(M)<\infty$ and $\gpd_R(M)<\infty$. Then
$\pd_R(M)=\gpd_R(M)<\infty$. Furthermore, for each $\p\in\spec(R)$
such that $\depth_{R_{\p}}(M_{\p})$ is finite, the localization
$C_{\p}$ is a dualizing 
$R_{\p}$-module. \qed
\end{thm}

\begin{cor} \label{cor0202}
Assume that $R$ is local, and 
let $C$ be a semidualizing $R$-module. The following conditions are equivalent:
\begin{enumerate}[\rm(i)]
\item \label{cor0202a}
$C$ is a dualizing $R$-module;
\item \label{cor0202b}
there exists a finitely generated $R$-module $M\neq 0$ such that
$\icid_R(M)<\infty$ and $\pd_R(M)<\infty$;
\item \label{cor0202c}
there exists an $R$-module $M\neq 0$ of finite depth
such that
$\icid_R(M)<\infty$ and $\gpd_R(M)<\infty$. \qed
\end{enumerate}
\end{cor}

\begin{disc} \label{disc0202}
As is noted in~\cite{holm:rfgid}, when $R$ has finite Krull dimension, 
we can change 
$\gpd_R(M)$ and $\pd_R(M)$  to $\gfd_R(M)$ and $\fd_R(M)$, respectively,
in the previous two results.
Similarly, in the next two results, if $\dim(R)<\infty$, then
$\gpcpd_R(M)$ and $\pcpd_R(M)$ can be changed to
$\gfcpd_R(M)$ and $\catf_C\text{-}\pd_R(M)$.
\end{disc}

\begin{thm} \label{thm0203}
Let 
$C$ be a semidualizing $R$-module, and 
let $M$ be an $R$-module
with $\id_R(M)<\infty$ and $\gpcpd_R(M)<\infty$. Then
$\pcpd_R(M)=\gpcpd_R(M)<\infty$. Furthermore, for each $\p\in\spec(R)$
such that $\depth_{R_{\p}}(M_{\p})$ is finite, the localization
$C_{\p}$ is a dualizing 
$R_{\p}$-module. \qed
\end{thm}

\begin{cor} \label{cor0204}
Assume that $R$ is local, and 
let $C$ be a semidualizing $R$-module. The following conditions are equivalent:
\begin{enumerate}[\rm(i)]
\item \label{cor0204a}
$C$ is a dualizing $R$-module;
\item \label{cor0204b}
there exists a finitely generated $R$-module $M\neq 0$ such that
$\id_R(M)<\infty$ and $\pcpd_R(M)<\infty$;
\item \label{cor0204c}
there exists an $R$-module $M\neq 0$ of finite depth
such that
$\id_R(M)<\infty$ and $\gpcpd_R(M)<\infty$. \qed
\end{enumerate}
\end{cor}

\begin{disc} \label{disc0201}
Holm proves his results in a more general setting than ours, namely,
over associative rings.
While the Gorenstein projective dimension and Gorenstein
injective dimension have been well-studied
in this setting, the same cannot be said for $\gc$-projective dimension and 
$\gc$-injective dimension.
Some of the foundation has been laid
by Holm and White~\cite{holm:fear}. To prove our results in this setting, though,
would require
a development of these ideas that is outside the scope of this paper.
\end{disc}

The next lemma is useful for the two subsequent proofs.

\begin{lem} \label{lem99}
Let $C$ be a semidualizing $R$-module.
If $\icid_R(C)<\infty$, then $C\cong R$ and $R$ is Gorenstein.
\end{lem}

\begin{proof}
Assume that $\icid_R(C)<\infty$.
Fact~\ref{projac} implies that $C\in\catacr$. By definition, this includes the condition
$\tor_{\geq 1}^R(C,C)=0$, and so~\cite[(3.8)]{frankild:rbsc} implies that
$C\otimes_RC$ is a semidualizing $R$-module. From~\cite[(3.2)]{frankild:sdcms}
we conclude that $C\cong R$. 
It follows that
$\id_R(R)=\icid_R(C)< \infty $
and so $R$ is Gorenstein as desired.
\end{proof}

\begin{disc} \label{disc:RTDW}
In unpublished work, Takahashi and White have proved the following result that
is weaker than Theorem~\ref{thm0002}:
If $R$ is Cohen-Macaulay with a dualizing module $D$ and 
$C$ is a semidualizing module with $\icid_R(C)<\infty$, then $C\cong D$.
\end{disc}

\begin{para} \label{para0202}
\emph{Proof of Theorem~\ref{thm0002}.}
Let $\p\in\supp_R(M)$, and replace $R$ with $R_{\p}$ to assume that $R$ is local.
In particular, every projective $R$-module is free, and so $M\cong C\oplus M'$
for some $M'\in\catpcr$. In the next sequence, the final equality
is from~\cite[(2.11.b)]{takahashi:hasm}
\begin{align*}
\sup\{\id_R(C\otimes_RC),\id_R(C\otimes_RM')\}
&=\id_R((C\otimes_RC)\oplus(C\otimes_RM'))\\
&=\id_R(C\otimes_R(C\oplus M'))\\
&=\id_R(C\otimes_RM)\\
&=\icid_R(M)
\end{align*}
and so
$\icid_R(C)\leq\icid_R(M)<\infty$. 
Lemma~\ref{lem99} implies that $R$ is Gorenstein, as desired.
\qed
\end{para}

The next result 
contains another partial answer to the question of Takahashi and White.
We include the proof because it is different from the proof of
Theorem~\ref{thm0002}.

\begin{thm} \label{thm0204}
If 
$C$ is a semidualizing $R$-module and $M$ is an $R$-module
such that $\pcpd_R(M)<\infty$ and $\icid_R(M)=0$, then
$R_{\p}$ is  Gorenstein for all $\p\in\spec(R)$ such that
$\depth_{R_{\p}}(M_{\p})$ is finite.
\end{thm}

\begin{proof}
As $M\in\caticr$, we have $M\cong\hom_R(C,E)$ for some injective $R$-module $E$.

We first show that the assumption that $\depth_{R_{\p}}(M_{\p})$ is finite implies that
$\p\in\ass_R(M)$. The fact that $C$ is finitely generated and $E$ is injective
yields the next isomorphisms
\begin{align*}
\ext^i_{R_{\p}}(R_{\p}/\p R_{\p},M_{\p})
&\cong\ext^i_{R_{\p}}(R_{\p}/\p R_{\p},\hom_R(C,E)_{\p})\\
&\cong\ext^i_{R_{\p}}(R_{\p}/\p R_{\p},\hom_{R_{\p}}(C_{\p},E_{\p}))\\
&\cong\hom_{R_{\p}}(\tor_i^{R_{\p}}(R_{\p}/\p R_{\p},C_{\p}),E_{\p}).
\end{align*}
Each module $\tor_i^{R_{\p}}(R_{\p}/\p R_{\p},C_{\p})$
is a finite-dimensional vector space over $R_{\p}/\p R_{\p}$. Furthermore,
we have $\tor_0^{R_{\p}}(R_{\p}/\p R_{\p},C_{\p})\cong
R_{\p}/\p R_{\p}\otimes_{R_{\p}}C_{\p}\neq 0$ since $C_{\p}$ is  nonzero
and finitely generated over $R_{\p}$.
Since $\ext^i_{R_{\p}}(R_{\p}/\p R_{\p},M_{\p})\neq 0$ for some $i$,
by hypothesis, it therefore follows that 
$\ext^0_{R_{\p}}(R_{\p}/\p R_{\p},M_{\p})\neq 0$,
and so $\p R_{\p}\in\ass_{R_{\p}}(M_{\p})$, that is,
$\p\in\ass_R(M)$ as claimed.

Write
$E\cong\oplus_{\q} E_R(R/\q)^{(\mu_{\q})}$, where the direct sum is taken
over all $\q\in\spec(R)$. It follows that there are equalities
$$\ass_R(M)=\supp_R(C)\cap\ass_R(E)=\spec(R)\cap\ass_R(E)
=\{\q\in\spec(R)\mid\mu_{\q}\neq 0\}.$$
Since $\p\in\ass_R(M)$, this implies $E\cong E_R(R/\p)\oplus E'$ for
some injective $R$-module $E'$. It follows that
$M\cong\hom_R(C,E_R(R/\p))\oplus\hom_R(C,E')$.
As in the proof of Theorem~\ref{thm0002}, using~\cite[(2.11.c)]{takahashi:hasm} we see that
$\pcpd_R(\hom_R(C,E_R(R/\p)))<\infty$. 
If follows that we may replace $R$ with $R_{\p}$
and $M$ with $\hom_R(C,E_R(R/\p))_{\p}$ to assume that
$R$ is local with maximal ideal $\m$ and 
$M\cong\hom_R(C,E)$ where $E=E_R(R/\m)$.

Let $\comp R$ denote the completion of $R$. 
It is straightforward to show that the condition 
$\pcpd_R(M)<\infty$ implies
$\catp_{\comp C}\text{-}\pd_{\comp R}(M\otimes_R\comp R)<\infty$.
Also, we have isomorphisms
\begin{align*}
M\otimes_R\comp R
\cong \hom_R(C,E_R(R/\m))\otimes_R\comp R 
\cong \hom_{\comp R}(\comp C,E_{\comp R}(\comp R/\m\comp R))
\end{align*}
and so $M\otimes_R\comp R\in\cati_{\comp C}(\comp R)$.
It follows that we may replace $R$ with $\comp R$ and
$M$ with $M\otimes_R\comp R$ to assume that $R$ is complete.

To complete the proof, we show that $\icid_R(C)<\infty$;
the desired conclusion then follows from Lemma~\ref{lem99}.
The module $M$ admits a bounded augmented $\catpc$-resolution
$$0\to C\otimes_R P_n\to\cdots\to C\otimes_R P_0\to M\to 0.$$
Applying the functor
$\hom_R(-,E)$
yields an exact sequence
$$0\to\underbrace{\hom_R(M,E)}_{\cong C}
\to \hom_R(C\otimes_R P_0,E)\to\cdots
\to \hom_R(C\otimes_R P_n,E)\to 0.$$
The isomorphism $\hom_R(M,E)\cong C$ follows from Matlis duality
because of the assumption $M\cong\hom_R(C,E)$.
Since each $P_i$ is projective, each module $\hom_R(P_i,E)$ 
is injective,
and so
$$\hom_R(C\otimes_R P_i,E)
\cong\hom_R(C,\hom_R(P_i,E))\in\catic(R).$$
It follows that the displayed exact sequence is an augmented
$\catic$-coresolution of $C$, and so $\icid_R(C)<\infty$, as desired.
\end{proof}

For our final result, recall that, when $R$ is local with maximal ideal $\m$,
the \emph{width} of an $R$-module $M$ is
$\width_R(M)=\inf\{i\geq 0\mid \tor_i^R(k,M)\neq 0\}$ where $k=R/\m$.

\begin{cor} \label{cor0205}
If 
$C$ is a semidualizing $R$-module and $M$ is an $R$-module
such that $\fcpd_R(M)=0$ and $\icid_R(M)<\infty$, then
$R_{\p}$ is  Gorenstein for all $\p\in\spec(R)$ such that
$\width_{R_{\p}}(M_{\p})$ is finite.
\end{cor} 

\begin{proof}
As in the proof of Theorem~\ref{thm0002}, replace $R$ and $M$ with $R_{\p}$ and $M_{\p}$
to assume that $R$ is local and that $\width_R(M)$ is finite.
It remains to show that $R$ is Gorenstein.
Let $E$ denote the injective hull of the residue field of $R$, and
set $(-)^{\vee}=\hom_R(-,E)$.
By assumption, there is a flat $R$-module $F$ such that $M\cong F\otimes_R C$.
Hom-tensor adjointness can be used to show that the $R$-module $F^{\vee}$
is injective, and so the sequence of isomorphisms
$$
M^{\vee}
\cong\hom_R(F\otimes_R C,E)
\cong \hom_R(C,\hom_R(F,E))
=\hom_R(C,F^{\vee})
$$
shows that $M^{\vee}\in\caticr$. 

We claim that $\pcpd_R(M^{\vee})<\infty$. To see this, consider a bounded
augmented $\catic$-coresolution of $M$
$$0\to M\to \hom_R(C,I^0)\to \hom_R(C,I^1)\to\cdots\to \hom_R(C,I^n)\to 0.$$
The functor $(-)^{\vee}$ yields an exact sequence
\begin{equation}
\label{cor0205a}\tag{$\dagger$}
0 \to \hom_R(C,I^n)^{\vee}\to\cdots
\to \hom_R(C,I^1)^{\vee}
\to \hom_R(C,I^0)^{\vee}
\to M^{\vee}\to 0.
\end{equation}
For each $j$, 
the module $(I^j)^{\vee}$ is flat by~\cite[(1.5)]{ishikawa:oimafm},
and Hom-evaluation~\cite[(1.6)]{ishikawa:oimafm}
explains the isomorphism in the next display
$$
\hom_R(C,I^j)^{\vee}
=\hom_R(\hom_R(C,I^j),E)
\cong C\otimes_R\hom_R(I^j,E)
=C\otimes_R(I^j)^{\vee}.$$
It follows that $\hom_R(C,I^j)^{\vee}\in\catfcr$, and so the sequence
\eqref{cor0205a} implies that $\fcpd_R(M^{\vee})$ is finite. 
Since $R$ is local, it follows from Fact~\ref{fact0101'} that
$\pcpd_R(M^{\vee})$ is also finite, as claimed.

Next, we claim that $\depth_R(M^{\vee})<\infty$. (Once this is shown,
it follows that $R$ is Gorenstein by Theorem~\ref{thm0204},
using the module $M^{\vee}$.) To verify the claim, it suffices to show that
$\ext^i_R(k,M^{\vee})\neq 0$ for some $i$.
The assumption
$\width_R(M)<\infty$ implies $\tor^R_i(k,M)$ is a nonzero $k$-vector space for some $i$.
Write $\tor^R_i(k,M)\cong\oplus_{\lambda\in\Lambda}k$ for some index set $\Lambda\neq\emptyset$.
The first isomorphism in the next sequence is a version of Hom-tensor adjointness
$$
\ext^i_R(k,M^{\vee})
\cong\tor^R_i(k,M)^{\vee} 
\cong(\oplus_{\lambda\in\Lambda}k)^{\vee}
\textstyle\cong\prod_{\lambda\in\Lambda}k^{\vee}
\cong\prod_{\lambda\in\Lambda}k
\neq 0
$$
and the remaining steps are standard.
\end{proof}

\section*{Acknowledgments}

S.S.-W. would like to thank the IPM for its generosity and hospitality during his visit
when these results were proved. Both authors are grateful to Diana White for her 
comments on an earlier version of this manuscript,
and to the anonymous referee for useful comments.


\providecommand{\bysame}{\leavevmode\hbox to3em{\hrulefill}\thinspace}
\providecommand{\MR}{\relax\ifhmode\unskip\space\fi MR }
\providecommand{\MRhref}[2]{%
  \href{http://www.ams.org/mathscinet-getitem?mr=#1}{#2}
}
\providecommand{\href}[2]{#2}

\end{document}